\documentclass[twoside,a4paper,11pt]{amsart}
\usepackage{a4,amsmath,amsthm,amsfonts,mathrsfs,latexsym,amssymb,pstricks,pst-grad,graphicx, tabularx,setspace,longtable,color}

\usepackage{amsmath,amsfonts,mathrsfs,latexsym,amssymb,pstricks,pst-grad, graphicx,tabularx, setspace,pstricks, graphics,enumerate,fancyhdr}
\usepackage{tikz}
\newtheorem{theorem}{Theorem}[section]

\newtheorem{proposition}[theorem]{Proposition}

\theoremstyle{definition}
\newtheorem{definition}[theorem]{Definition}
\newtheorem{rem}[theorem]{Remark}
\newtheorem{example}[theorem]{Example}
\newtheorem{question}[theorem]{Question}

\usepackage{enumerate}
\usepackage{color}
\usepackage{psfrag}
\usepackage{amssymb}
\usepackage{mathtools}
\scshape
\date{}
\usepackage{tikz}
\usetikzlibrary{arrows}
\usepackage{subcaption}
\usetikzlibrary{automata}

\newcommand{\Irr}{\mathop{\mathrm{Irr}}}
\newcommand{\cd}{\mathop{\mathrm{cd}}}
\newcommand{\adim}{\mathop{\mathrm{adim}}}
\newcommand{\diam}{\mathop{\mathrm{diam}}}

\begin{document}
\title[Metric dimension of character degree graphs]{On the metric dimension of the character degree graph of a solvable group}

\author [Cameron]{Peter J. Cameron}
\address{School of Mathematics and Statistics, University of St Andrews, North Haugh St Andrews, Fife, KY16 9SS, U.K.,ORCID:0000-0003-3130-9505}
\email{pjc20@st-andrews.ac.uk}
\author [Sivanesan]{G. Sivanesan}
\address{Department of Mathematics, Government College of Engineering, Salem 636011, Tamil Nadu, India, ORCID:0000-0001-7153-960X.}
\email{sivanesan@gcesalem.edu.in}
\author [Selvaraj]{C. Selvaraj}
\address{Department of Mathematics, Periyar University, Salem 636011, Tamil Nadu, India, ORCID:0000-0002-4050-3177.}
\email{selvavlr@yahoo.com}
\author[Tamizh Chelvam]{T. Tamizh Chelvam}
\address{Department of Mathematics, Manonmaniam Sundaranar University, 
Tirunelveli 627 012, Tamil Nadu, India, ORCID:0000-0002-1878-7847.}
\email{tamche59@gmail.com}
\author[Laubacher]{Jacob Laubacher}
\address{Department of Mathematics, St. Norbert College, De Pere, WI 54115, ORCID:0000-0003-0045-7951.}
\email{jacob.laubacher@snc.edu}

 
\begin{abstract}
Let $G$ be a finite solvable group and let $\Delta(G)$ be the character degree graph of $G$. In this paper, we obtain the metric dimension of certain character degree graphs. Specifically, we calculate the metric dimension for a regular character degree graph, a character degree graph with a diameter of $2$ that is not a block, a character degree graph with a diameter of $3$ that also has a cut vertex and a character degree graph with Fitting height $2.$ We also consider two related parameters, base size and adjacency dimension, and their relation to metric dimension for character degree graphs of solvable groups.
\end{abstract}
\date{}
\maketitle
 
\section{Introduction}
Throughout this article, we let $G$ be a finite solvable group with identity $1$. We fix $\Irr(G)$ to be the set of all irreducible characters of $G$ and $\cd(G)=\{\chi(1)\mid\chi\in \Irr(G)\}$. Letting $\rho(G)$ be the set of all primes that divide degrees in $\cd(G)$, one can then define the character degree graph of $G$, denoted by $\Delta(G)$. This simple, undirected graph has the vertex set $\rho(G)$, and there is an edge between two  distinct primes $p$ and $q$ if $pq$ divides some degree $a\in\cd(G)$. This graph $\Delta(G)$ was first defined by Manz \textit{et al.}\ in \cite{Manz2}, and a broad overview of character degree graphs can be found in \cite{Lewis5}. Recently, a lot of work has been done concerning these graphs $\Delta(G)$. For example, Manz proved that the graph contains at most two connected components in \cite{Manz1}, and furthermore, Wolf \textit{et al.}\ in \cite{Manz3} prove that the sum of the diameters of the components of $\Delta(G)$ is at most $3$. Classifying whether a certain graph occurs as the prime character degree graph of a solvable group has been of interest (see \cite{BissLaub, LJ, LM}, for example), and in \cite{EI} Ebrahimi et al. proved that the character degree graph $\Delta(G)$ of a solvable group $G$ is Hamiltonian if and only if $\Delta(G)$ contains a block with at least three vertices. Finally, Sivanesan et al. in \cite{SST} provide a necessary condition for the character degree graph of a finite solvable group to be Eulerian.

(We note that, in some of the literature on this subject, the character degree graph is denoted by $\Gamma(G)$, while $\Delta(G)$ denotes the graph whose vertex set is $\cd(G)$, two vertices adjacent if they have a common prime divisor.)

In recent decades, a considerable amount of research has been done on the metric dimension of graphs. In 1976, Harary and Melter were the first to study the problem of finding a graph's metric dimension (see \cite{Har}). The paper~\cite{BC} gives a summary of metric dimension and related notions such as determining number. More recently, Pirzada and Raja examine the metric dimension of zero-divisor graphs in \cite{PSR}. The purpose of this paper is to investigate the metric dimension of the character degree graph $\Delta(G)$ for some solvable group $G$.

In Section \ref{SecPre}, we provide all the necessary background information, attempting to keep this paper as self-contained as possible. We recall all the needed foundations of basic graph theory, including diameter and dimension, among others. In Section \ref{SecMet}, we present our main results, along with accompanying examples for further illustration. In Section \ref{FitHei}, we discuss the character degree graphs of finite groups with Fitting height 2 using Lewis partition. The final section gives a brief discussion of two related concepts, base size and adjacency dimension, and study their relation to metric dimension for character degree graphs of solvable groups. Depending on certain properties of the character degree graph $\Delta(G)$, we prove that we can calculate the metric dimension. In particular, we obtain the metric dimension for when the character degree graph: is regular (Theorem \ref{3.1}), has a diameter of $2$ but is not a block (Theorem \ref{3.2}), has a diameter of $3$ but contains a cut vertex (Theorem \ref{3.5}) and has Fitting height $2$ (Theorem \ref{3.4}). 

\section{Preliminaries}\label{SecPre}
Throughout this paper, when we refer to a graph $\Gamma$, we mean a simple, undirected finite graph $\Gamma$ with vertex set $V(\Gamma)$ and edge set $E(\Gamma)$. A clique in $\Gamma$ is a complete induced subgraph of $\Gamma$. We will also commonly refer to the graph $K_n$, which is the complete graph on $n$ vertices, and the graph $C_n$, which is the cycle graph on $n$ vertices. For all other foundational graph theoretic concepts, we refer the reader to \cite{BM}, and we refer to \cite{HALL} for basic notations and terminology in group theory. For less commonly known definitions and results, we present here a short overview.


\begin{definition}
A \textbf{cut vertex} of a graph $\Gamma$ is a vertex whose removal increases the number of components.
\end{definition}

\begin{definition}
A connected nontrivial graph having no cut vertex is called a \textbf{block}.
\end{definition}

It is then natural to view blocks as subgraphs of a given graph $\Gamma$. It is well known that any two blocks of $\Gamma$ have at most one vertex in common.

\begin{theorem}\cite[Lemma 2.7]{EI}\label{2.1}
Let $G$ be a group with $|\rho(G)|\geq 3$.  If $\Delta(G)$ is not a block and the diameter of $\Delta(G)$ is at most $2$, then each block of $\Delta(G)$ is a complete graph.
\end{theorem}

Next, we look at the diameter of a graph. For a finite solvable group $G$, it is well established that $\diam(\Delta(G))\leq 3$ (see \cite[Theorem 3.2]{Manz3}). Lewis proved in \cite{Lewis2} that in order for $\diam(\Delta(G))=3$, then we must have that $|\rho(G)|\geq6$. In fact, he provided an example of such a group whose graph has six vertices in \cite{Lewis3}.

\begin{theorem}\cite[Theorem 10]{HHH}\label{2.5} 
If $G$ is a solvable group and the character degree graph $\Delta(G)$ is connected, then $\Delta(G)$ has at most one cut vertex. In particular, if  $v$ is the unique cut vertex of $\Delta(G)$, then $\diam(\Delta(G))=3$ if and only if $\Delta(G)$ has one of the following structures:
\begin{itemize}
\item[\rm (1)]$\Delta(G)=w-v-K_s$, for some $w\in\rho(G)$ and some integer $s\ge4$ such that $V(K_s)\cup\{v\}$ does not generate a clique.
\item[\rm (2)] $\Delta(G)=K_m-v-K_s$, for some integers $m\ge8$ and $ s\ge2$, where $V(K_s)\cup\{v\}$ generates a clique and $1< |N(v)|\cap V(K_m)\mid<m$.
\end{itemize}
Furthermore, in both structures $(1)$ and $(2)$, $deg(v)\ge4$.
\end{theorem}

\begin{rem}\label{rem2.1}
When $\diam(\Delta(G))=3$, we can partition the vertex set $\rho(G)$ into four disjoint nonempty sets. This was formally done by Sass in \cite{Sass}, where $\rho(G)=\rho_1\cup\rho_2\cup\rho_3\cup\rho_4$ under the following rules: since $\diam(G)=3$, then there must exist vertices $r,s\in\rho(G)$ such that the distance between $r$ and $s$ is $3$. So we denote $\rho_4$ to be the set of all vertices of $\Delta(G)$ which are at distance $3$ from the vertex $r$. As consequence, we know that $s\in \rho_4$. Next, $\rho_3$ is  the set of all vertices of $\Delta(G)$ which are distance $2$ from the vertex $r$. It is established that $|\rho_3|\geq3$, for instance (again, see \cite{Sass}). Then we have that $\rho_2$ is the set of all vertices that are adjacent to the vertex $r$ and adjacent to some prime in $\rho_3$, whereas $\rho_1$ consists of $r$ and the set of all vertices which are adjacent to $r$ and not adjacent to any vertex in $\rho_3.$

The above implies that no vertex in $\rho_1$ is adjacent to any vertex in $\rho_3\cup\rho_4$, and no vertex in $\rho_4$ is adjacent to any vertex in $\rho_1\cup\rho_2$. Furthermore, every vertex in $\rho_2$ is adjacent to some vertex in $\rho_3$ and vice versa, and $\rho_1\cup\rho_2$ and $\rho_3\cup\rho_4$ both induce complete subgraphs of $\Delta(G)$.
\end{rem}

Using this, we are now ready for the following recent result by Lewis and Meng:

\begin{theorem}\cite[Lemma 2.1]{Lewis6}\label{2.3}
Let $G$ be a solvable group and assume that $\Delta(G)$ has diameter $3$. Then $G$ is $1$-connected if and only if $|\rho_2|=1$ in the diameter $3$ partition of  $\rho(G)$. In this case, if $p$ is the unique prime in $\rho_2$, then $p$ is also the unique cut vertex for $\Delta(G)$. In particular, $\Delta(G)$ has at most one cut vertex.
\end{theorem}

\begin{definition}
The \textbf{connectivity}  $\kappa=\kappa(\Gamma)$ of a graph $\Gamma$ is the minimum number of vertices whose removal gives a disconnected or trivial graph.
\end{definition}

Further graph-theoretic background we need concerns twin vertices; this will be
discussed in the final section of the paper.

\begin{definition}
For a subset $W= \{v_1,v_2,\ldots,v_k\}\subseteq V(\Gamma)$ and for the vertex $v\in V(\Gamma)$,  the vector $r(v\mid W)=(d(v,v_1), d(v,v_2),\ldots,d(v,v_k))$ is the \textbf{metric $W$-representation} of $v$.
\end{definition}

A subset  $W\subseteq V(\Gamma)$ is said to be a metric generator for $\Gamma$ if all the vertices of $\Gamma$ have pairwise different metric $W$-representations. The metric basis for $\Gamma$ is a metric generator of the smallest order, whose order is denoted as dim$(\Gamma)$. This is the focus of our work in this paper. Chartrand et al. in \cite{CE} proved that dim$(\Gamma)=1$ if and only if $\Gamma=P_n,$ the path with $n$ vertices. Furthermore, dim$(\Gamma)=2$ if and only if $\Gamma=C_n$, and dim$(\Gamma)=n-1$ if and only if $\Gamma=K_n$ for all $n\geq2$.

\begin{theorem}\cite[Theorem 4]{CE}\label{2.6}
Let $G$ be a connected graph of order $n \ge 4.$ Then $dim(G)=n-2$ if and only if  $G = K_{s,t} (s, t \ge 1)$, $G = K_s +$  $\overline{K_t}$ $(s\ge 1, t \ge 2)$ or $G = K_s +(K_1 \cup K_t)$ $(s,t \ge 1).$
\end{theorem}

Finally, we recall a result about regular graphs, which we use in Section \ref{SecMet}.

\begin{theorem}\cite[Theorem A]{MZ}\label{2.2}
If $\Delta(G)$ is a non-complete and regular character degree graph of a finite solvable group $G$ with $n$ vertices, then $\Delta(G)$ is a $(n-2)$-regular graph.
\end{theorem}

The related concepts of base and adjacency resolving set, and the corresponding
dimensions, are defined in the final section.

\section{Metric Dimension of  Character Degree Graphs}\label{SecMet}
Three of the authors of this paper earlier proved that, for $n\geq4$, every $(n-2)$-regular graph is the character degree graph $\Delta(G)$ for some solvable group $G$ (see \cite{SST}). In the result that follows, we obtain the metric dimension of $(n-2)$-regular character degree  graphs. 

\begin{theorem}\label{3.1}
Let $G$ be a solvable group of order $n\geq4$ and $n$ even. Assume that the  character degree graph $\Delta(G)$ is  non-complete and regular. Then the metric dimension of  $\Delta(G)$ is $\frac{n}{2}$.
\end{theorem}

We defer the proof of this theorem until Section~\ref{s:further}, where it will
follow from a general graph-theoretic result.

\begin{example}
Since $\Gamma$ is a $4$-regular graph on six vertices, there exists a solvable group $G$ such that $\Delta(G)\cong \Gamma$ (see \cite{SST}). One can check that $W=\{v_1,v_2,v_3\}$ is a minimum resolving set of $\Gamma$, and so its metric dimension is $\frac{n}{2}=\frac{6}{2}=3$. See Figure \ref{f1} for a visual.
\end{example}

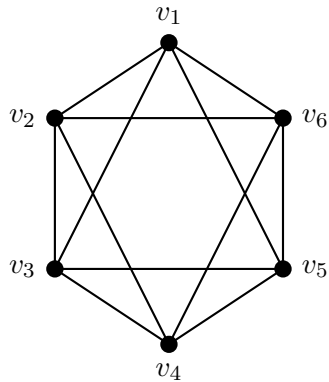
\begin{figure}[ht]
\centering
\begin{tikzpicture}[scale=1,colorstyle/.style={circle, draw=black!100,fill=black!100, thick, inner sep=0pt, minimum size=2mm}]
\node at (2,3)[colorstyle,label=above:$v_1$]{};
\node at (0.5,2)[colorstyle,label=left:$v_2$]{};
\node at (0.5,0)[colorstyle,label=left:$v_3$]{};
\node at (2,-1)[colorstyle,label=below:$v_4$]{};
\node at (3.5,0)[colorstyle,label=right:$v_5$]{};
\node at (3.5,2)[colorstyle,label=right:$v_6$]{};
\draw[thick] (2,3) --(0.5,2)--(0.5,0)--(2,-1)--(3.5,0)--(3.5,2)--(2,3);
\draw[thick] (2,3) --(0.5,0)--(3.5,0)--(2,3);
\draw[thick] (0.5,2) --(2,-1)--(3.5,2)--(0.5,2);
\end{tikzpicture} 
\caption{A $4$-regular character degree graph with six vertices}
\label{f1}
\end{figure}

Now, we obtain the metric dimension of a particular class of character degree graphs. Before proceeding in general, however, we give an example.

\begin{example}\label{e3.2}
Consider the graph shown in Figure \ref{f3}. It is a character degree graph of a solvable group $G$ with six vertices (see\cite[p: 503 ]{BL}). Observe that $\Delta(G)$ consists of two complete blocks. The vertices of one block are not adjacent to those of the other block, but there is one cut vertex adjacent to all others. Here $W=\{v_2,v_3,v_5\}$ is a minimum resolving set of $\Delta(G)$ and so its metric dimension is $n-3=6-3=3$.
\end{example}

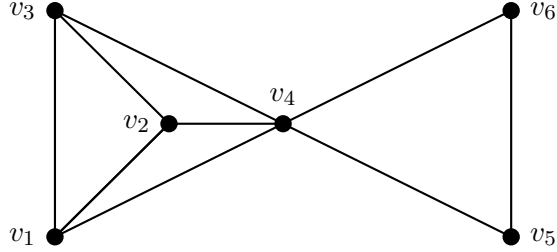
\begin{figure}[ht]
\centering
\begin{tikzpicture} [scale=1.5,colorstyle/.style={circle, draw=black!100,fill=black!100, thick, inner sep=0pt, minimum size=2mm}]            
\node at (0,0)[colorstyle,label=left:$v_1$]{};
\node at (1,1)[colorstyle,label=left:$v_2$]{};
\node at (0,2)[colorstyle,label=left:$v_3$]{};
\node at (2,1)[colorstyle,label=above:$v_4$]{};
\node at (4,0)[colorstyle,label=right:$v_5$]{};
\node at (4,2)[colorstyle,label=right:$v_6$]{};
\draw[thick] (0,0) --(0,2) --(2,1) --(4,2);
\draw[thick] (0,0) --(2,1) --(4,0) --(4,2);
\draw[thick] (0,0) --(1,1) --(2,1) ;
\draw[thick] (0,0) --(1,1) --(0,2) ;
\end{tikzpicture} 
\caption{Character degree graph has $2$ blocks of diameter $2$}
\label{f3}
\end{figure}

We now obtain the metric dimension of a class of character degree graphs. In fact we consider the class of character degree graphs exhibited in Example~\ref{e3.2}. By Theorem~\ref{2.1}, if $\Delta(G)$ is not a block and the diameter of $\Delta(G)$ is at most $2,$ then each block of $\Delta(G)$ is a complete graph. For a graph $\Gamma$ and a vertex $v\notin V(\Gamma)$, by $\Gamma - v$, we mean a graph obtained by joining every vertex of $\Gamma$ with $v$.

\begin{theorem}\label{3.2}  Assume that $G$ is a solvable group. Let $\Delta(G)$ be the character degree graph of $G$ with $n$ vertices and $\Delta(G)=K_{n-m-1}-v-K_m$ for some integer $m$ and $1 \leq m \le n-m-1$. Then the metric dimension of $\Delta(G)$ is $n-3$.
\end{theorem} 

\begin{proof} Assume that $\Delta(G) = K_{n-m-1}-v-K_m$ be the character degree graph of $G$ with $n$ vertices. Since the complete graph $K_m$ is not adjacent with any vertex of $ K_{n-m-1},$ the diameter of $\Delta(G)$ is $2.$ Further $V(K_{n-m-1}) \cup \{v\}$ and $V(K_m) \cup \{v\}$ generate cliques of $\Delta(G).$ Hence $\Delta(G)$ is not a block.  Let $V(K_m) = \{v_1,v_2, v_3,\ldots,v_m\}, v_{m+1}= v$ and  $V(K_{n-m-1})=\{v_{m+2},v_{m+3},\ldots,v_{n}\}$ Hence $V(\Delta(G))=\{v_1,\ldots,v_n\}.$ Let $W=V(\Delta(G))\setminus \{v_1,v_{m+1},v_n\}.$ Note that the jth coordinate of $r(v_j|W)$ is zero and all other coordinates of $r(v_j|W)$ are positive. Further, we have $r(v_i|W)=r(v_j|W)$ implies that $v_i=v_j$ for $i\neq j.$ Thus $r(v_j|W)$ are distinct for $2\leq j \leq n-1$ and $j\neq m+1.$ 

Further $r(v_1|W_1)=(\underbrace{1,1,\ldots,1,}_{m-1\hspace{0.2cm} times}\underbrace{2,2,\ldots,2}_{n-m-2\hspace{0.2cm}times})$, 

$r(v_{m+1}|W_1)=(\underbrace{1,1,\ldots,1,}_{m-1\hspace{0.2cm} times}\underbrace{1,1,\ldots,1}_{n-m-2\hspace{0.2cm} times})$ and 

$r(v_n|W_1)=(\underbrace{2,2,\ldots,2,}_{m-1\hspace{0.2cm} times} \underbrace{1,1,\ldots,1}_{n-m-2\hspace{0.2cm} times}).$  

Consequently $W$ is a resolving set for $\Delta(G).$ 

\textbf{Claim.} $W$ is a minimum resolving set for $\Delta(G).$ Suppose that $W_1\subset V,$ $|W_1|\leq n-4$ and $W_1$ is a resolving set for $\Delta(G).$ Without loss of generality one can take as $|W_1|=n-4.$ Then there exist four elements $v_i, v_j, v_k, v_{\ell}\in V\setminus W_1.$
As per the structure of $\Delta(G),$ at least one of the following cases shall take place.

\textbf{Case 1.} Two elements out of four elements $v_i,v_j,v_k,v_{\ell}$ are in $V(K_m).$  Without loss of generality, let us take $v_i,v_j\in V(K_m).$
Then

$r(v_i|W_1) = (\underbrace{1,1,\ldots,1,}_{m-2\hspace{0.2cm} times}1, \underbrace{2, 2,\ldots ,2}_{n-m-3\hspace{0.2cm} times})$ and 

$r(v_j |W_1) = (\underbrace{1,1,\ldots,1,}_{m-2\hspace{0.2cm} times}1,  \underbrace{2, 2,\ldots ,2}_{n-m-3\hspace{0.2cm} times}).$ From this $W_1$ is not a resolving set, a contradiction.

\textbf{Case 2.} Two elements out of four elements $v_i,v_j,v_k,v_{\ell}$ are in $V(K_{n-m-1}).$ Without loss of generality, let us take $v_i,v_j\in V(K_{n-m-1}).$ Then
$r(v_i|W_1) = (\underbrace{2,2,\ldots,2,}_{m-2\hspace{0.2cm} times}1, \underbrace{1, 1,\ldots ,1}_{n-m-3\hspace{0.2cm} times})$ and 
$r(v_j |W_1) = (\underbrace{2,2,\ldots,2,}_{m-2\hspace{0.2cm} times}1,  \underbrace{1, 1,\ldots ,1}_{n-m-3\hspace{0.2cm} times}).$ From this $W_1$ is not a resolving set, a contradiction.

Hence the metric dimension of $\Delta(G)$ is $n-3.$  
\qed
\end{proof} 

In Theorem \ref{2.5}, Hafezieh et al. (see \cite[Theorem 10]{HHH}) obtained the structure of the character degree graph with a cut vertex and of diameter $3.$ Using this structure, now we obtain the metric dimension.

\begin{theorem}\label{3.5} Let $G$ be a finite solvable group and $\Delta(G)$ be the character degree graph of $G$ with $n$ vertices. Assume that the diameter of $\Delta(G)$ is $3$ with Lewis' partition $\rho(G)=\rho_1 \cup \rho_2 \cup \rho_3\cup \rho_4.$ Suppose that $\Delta(G)$ contains a cut vertex, then one of the following is true.
\begin{itemize}
\item[\rm (1)] If $|\rho_1|=1,$ then the metric dimension of $\Delta(G)$ is $n-3.$
\item[\rm (2)] If $|\rho_1| \geq 2$, then the metric dimension of $\Delta(G)$ is $n-4.$
\end{itemize}
\end{theorem}
\begin{proof} Since $\Delta(G)$ has diameter three, as mentioned in Remark~\ref{rem2.1} the primes in $\rho_1$ are not adjacent to any prime in $\rho_3 \cup \rho_4$, and the primes in $\rho_4$ are not adjacent to any prime in $\rho_1 \cup \rho_2$. Both $\rho_1 \cup \rho_2$ and $\rho_3 \cup \rho_4$ induce complete subgraphs of $\Delta(G)$. By Theorem~\ref{2.3}, $|\rho_2|=1.$   Let us label the vertices of $\rho_1,\rho_2,\rho_3$ and $\rho_4$ as $\rho_1=\{u_1,u_2,\ldots,u_{n_1}\}, \rho_2=\{v_1\},\rho_3=\{w_1,w_2,\ldots,w_{n_3}\}$ and $\rho_4=\{x_1,x_2,\ldots,x_{n_4}\}.$ Thus $n=n_1+1+n_3+n_4.$ 
 
(1). Assume that $\rho_1=\{u_1\}.$ Now, we have $V(\Delta(G))=\{u_1,v_1,w_1,\linebreak w_2,\ldots,w_{n_3},x_1,x_2,\ldots, x_{n_4}\}.$ 

\textbf{Claim.} $X=V(\Delta(G))\setminus \{u_1,v_1,x_{n_4}\}$ is a minimum resolving set for $\Delta(G).$ 

For, the jth coordinate of  $r(w_j|W)$ are distinct for $1\leq j \leq n_3$ and $r(x_j|W)$ are distinct for $1\leq j\leq n_4-1$. It remains to consider the vertices $u_1,v_1,x_{n_4}.$ 

Note that $r(u_1|X)=(\underbrace{2,2,\ldots,2}_{n_3 \hspace{0.2cm} times},\underbrace{3,3,\ldots,3}_{n_4-1 \hspace{0.2cm} times})$,

$r(v_1|X) = ( \underbrace{1,1,\ldots,1}_{n_3 \hspace{0.2cm} times},\underbrace{2,2,\ldots,2}_{n_4-1 \hspace{0.2cm} times})$ and 

$r(x_{n_4}|X) = (\underbrace{1, 1,\ldots,1}_{n_3 \hspace{0.2cm} times},\underbrace{1,1,\ldots,1}_{n_4-1 \hspace{0.2cm} times})$.
 
It follows that $X$ is a resolving set for $\Delta(G)$. We shall now proceed to prove that $X$ is a minimum resolving set for $\Delta(G)$. Suppose $X_1\subset V(\Delta(G)),$ $|X_1|\leq n-4$ and $X_1$ is a resolving set. Without loss of generality let us take $|X_1|=n-4.$  Thus there exist four elements in $\rho (G)$ which are not in $X_1.$

\textbf{Case 1.} Suppose there exist $w_i,w_j\in\rho_3$ with $w_i,w_j\notin X_1.$ Then the metric representations $r(w_i|X_1)$ and $r(w_j|X_1)$ are same, a contradiction to $X_1$ is a resolving set.

\textbf{Case 2.} Suppose there exist $x_i,x_j\in\rho_4$ and  $x_i,x_j\notin X_1.$
Then the metric representations $r(x_i|X_1)$ and $r(x_j|X_1)$ are same, a contradiction to $X_1$ is a resolving set.

\textbf{Case 3.} Suppose there exist $w_i\in \rho_3,$ and $x_j\in \rho_4$ and  $u_1,v_1,w_i,x_j \notin X_1$. This implies that $X_1\subset \rho_3\cup\rho_4.$ As mentioned in Remark~\ref{rem2.1}, $\rho_3\cup\rho_4$ induce a complete subgraph of $\Delta(G).$ Then the metric representations $r(w_i|X_1)$ and $r(x_j|X_1)$ are same, a contradiction to $X_1$ is a resolving set.

Therefore, $X_1$ is not a resolving set. Hence $X=V(\Delta(G))\setminus \{u_1,v_1,x_{n_4}\}$ is a minimum resolving set for $\Delta(G)$ and so the metric dimension is $n-3.$

(2). Assume that $|\rho_1|\geq 2.$ Now, $V(\Delta(G))= \{u_1,u_2,\ldots,u_{n_1},v_1, w_1,w_2,\ldots,\linebreak w_{n_3}, x_1,x_2,\ldots,x_{n_4}\}$. 

\textbf{Claim.} $X=V(\Delta(G))\setminus \{u_1, v_1, w_{n_3},  x_{n_4}\}$ is a minimum resolving set for $\Delta(G).$ For, the jth coordinate of $r(u_j|X)$ is $0$, while the jth coordinate of all other vertex representations is positive. Thus, it is clear that $r(u_i|X) = r(u_j|X)$ implies that $u_i=u_j$. Therefore $r(u_j|W)$ are distinct for $2\leq j \leq n_1$. In the same way $r(w_j|W)$ are distinct for $1\leq j \leq n_3-1$ and $r(x_j|W)$ are distinct for $1\leq j\leq n_4-1$. We need to consider the vertices $u_1,v_1,w_{n_3},x_{n_4}.$ 

Note that $r(u_1|X)=(\underbrace{1,\ldots,1,}_{n_1-1\hspace{0.2cm} times}\underbrace{2,2,\ldots,2}_{n_3-1 \hspace{0.2cm} times},\underbrace{3,3,\ldots,3}_{n_4-1 \hspace{0.2cm} times})$,

$r(v_1|X) = (\underbrace{1,\ldots,1,}_{n_1-1\hspace{0.2cm} times} \underbrace{1,1,\ldots,1}_{n_3-1 \hspace{0.2cm} times},\underbrace{2,2,\ldots,2}_{n_4-1 \hspace{0.2cm} times})$

$r(w_{n_3}|X) = (\underbrace{2,2,\ldots,2,}_{n_1-1\hspace{0.2cm} times} \underbrace{1, 1,\ldots,1}_{n_3-1 \hspace{0.2cm} times},\underbrace{1,1,\ldots,1}_{n_4-1 \hspace{0.2cm} times})$ and 

$r(x_{n_4}|X) = (\underbrace{3,3,\ldots,3,}_{n_1-1\hspace{0.2cm} times} \underbrace{1, 1,\ldots,1}_{n_3-1 \hspace{0.2cm} times},\underbrace{1,1,\ldots,1}_{n_4-1 \hspace{0.2cm} times})$.
 
It follows that $X$ is a resolving set for $\Delta(G)$. We shall now proceed to prove that $X$ is a minimum resolving set for $\Delta(G)$. Suppose $X_1\subset V(\Delta(G)),$ $|X_1|\leq n-5$ and $X_1$ is a resolving set. Without loss of generality let us take $|X_1|=n-5.$  As a result, there exist five elements in $V(\Delta(G))$ which are not in $X_1.$ Now, one of the following cases is true.

\textbf{Case 1.} Suppose there exist $u_i,u_j\in\rho_1$ and $u_i,u_j\notin X_1.$ Then the
metric representations $r(u_i|X_1)$ and $r(u_j|X_1)$ are same, a contradiction to the statement that $X_1$ is a resolving set.

\textbf{Case 2.} Suppose there exist $w_i, w_j\in\rho_3$ and  $w_i,w_j \notin X_1.$ Then the metric representations $r(w_i|X_1)$ and $r(w_j|X_1)$ are same, a contradiction to $X_1$ is a resolving set.

\textbf{Case 3.} Suppose there exist $x_i,x_j \in \rho_4$ and  $x_i, x_j \notin X_1.$
Then the metric representations $r(x_i|X_1)$ and $r(x_j|X_1)$ are same, a contradiction to $X_1$ is a resolving set.

Therefore, $X_1$ is not a resolving set. Hence $X=V(\Delta(G))\setminus \{u_1,v_1,  w_{n_3}, x_{n_4}\}$ is a minimum resolving set for $\Delta(G)$ and so the metric dimension is $n-4.$
\qed
\end{proof}

\begin{example}\label{e4.1}
Consider the graph given in Figure \ref{f4}. It is a graph with diameter $3$, and so is a character degree graph $\Delta(G)$ of a solvable group $G$ with a cut vertex (see \cite[p: 629]{Lewis3}). Further $\rho_1=\{v_1 \},\rho_2=\{v_2 \},\rho_3=\{v_3,v_4,v_5 \}$ and $\rho_4=\{v_6 \}.$ Clearly, $\rho_1 \cup \rho_2$ and $\rho_3\cup\rho_4$ induce complete subgraphs of $\Delta(G).$ Further note that $X=\{v_3,v_4,v_5\}$ is a minimum  resolving set of $\Delta(G)$ and so its metric dimension is $n-3=6-3=3$.
\end{example}

\begin{figure}[ht]
\centering
\begin{tikzpicture} [scale=2,colorstyle/.style={circle, draw=black!100,fill=black!100, thick, inner sep=0pt, minimum size=2mm}]            
\node at (1,0)[colorstyle,label=right:$v_1$]{};
\node at (0,0)[colorstyle,label=above:$v_2$]{};
\node at (-0.5,0)[colorstyle,label=above:$v_3$]{};
\node at (-1.5,0.75)[colorstyle,label=above:$v_4$]{};
\node at (-1.5,-0.75)[colorstyle,label=right:$v_5$]{};
\node at (-2.5,0)[colorstyle,label=left:$v_6$]{};
\draw[thick] (1,0) --(0,0) ;
\draw[thick] (0,0) --(-1.5,0.75) --(-2.5,0) -- (-1.5,-0.75)--(0,0);
\draw[thick] (0,0) -- (-0.5,0) --(-1.5,0.75) ;
\draw[thick] (-0.5,0) --(-2.5,0);
\draw[thick]  (-0.5,0)-- (-1.5,-0.75)-- (-1.5,0.75);
\end{tikzpicture} 
\caption{Character degree graph with diameter $3$ and a cut vertex}
\label{f4}
\end{figure}
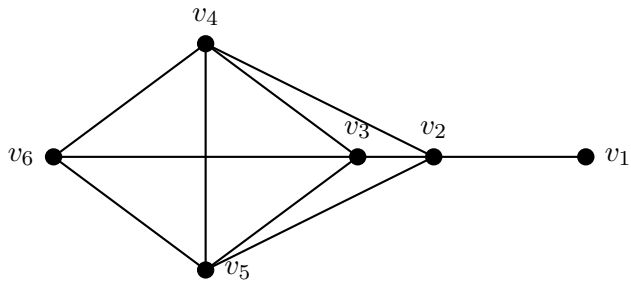

\section{Groups with Fitting height $2$}\label{FitHei}
Recall that a finite group $G$ has Fitting height~$2$ if $G/F(G)$ is nilpotent,
where the Fitting subgroup $F(G)$ is the largest nilpotent normal subgroup of
$G$.

Our discussion of the character degree graphs of finite groups with Fitting
height~$2$ is helped greatly by the theorem of Mark Lewis~\cite{Lewis4}:

\begin{theorem}\cite[Theorem A]{Lewis4}\label{2.4}
Let $\Gamma$ be a graph with $n$ vertices. There exists a solvable group $G$ of Fitting height $2$ with  $\Delta(G)=\Gamma$ if and only if the vertices of degree less than $n-1$ can be partitioned into two subsets, each of which induces a complete subgraph of $\Gamma$ and one of which contains only vertices of degree $n-2$.
\end{theorem}

We will call a graph having these properties a \textbf{Lewis graph}, and we
introduce some notation to describe these graphs. 

Let $\Gamma$ be a Lewis graph on $n$ vertices. We let $U$ be the set of
vertices of degree $n-1$ (the dominating or universal vertices), and $X$ and
$Y$ the two sets in Lewis' theorem, where $X$ contains only vertices of
degree $n-2$. The cardinalities of $U$, $X$, $Y$ will be denoted by $n_1$,
$n_2$ and $n_3$, where $n_1+n_2+n_3=n$, the number of vertices of $\Gamma$.

We note that $n_3\le n_2$, since non-adjacency induces a surjective (but not
necessarily injective) map from $X$ to $Y$.

It is easier to visualize the complement $\bar\Gamma$ of such a
graph $\Gamma$. In $\bar\Gamma$, the vertices in $U$ are the isolated vertices.
The edges of the graph have one vertex in $X$ and one in $Y$; each vertex
in $X$ lies on just one such edge. So $\bar\Gamma$ consists of $n_1$ isolated
vertices and a number $r$ of stars, where the central vertex of each star
lies in $Y$. Figure~\ref{f5} shows a typical such graph.

\begin{figure}[htbp]
\begin{center}
\setlength{\unitlength}{1.2mm}
\begin{picture}(100,15)
\multiput(20,0)(10,0){3}{\line(0,1){10}}
\multiput(20,0)(10,0){3}{\circle*{1}}
\multiput(20,10)(10,0){3}{\circle*{1}}
\multiput(20,10)(10,0){3}{\circle{2}}
\multiput(50,0)(20,0){3}{\circle*{1}}
\multiput(45,10)(10,0){6}{\circle*{1}}
\multiput(50,0)(20,0){3}{\line(-1,2){5}}
\multiput(50,0)(20,0){3}{\line(1,2){5}}
\put(90,0){\line(0,1){10}}
\put(90,10){\circle*{1}}
\multiput(45,10)(20,0){3}{\circle{2}}
\multiput(85,10)(5,0){2}{\circle{2}}
\multiput(0,10)(10,0){2}{\circle*{1}}
\multiput(0,0)(10,0){2}{\circle*{1}}
\multiput(0,10)(10,0){2}{\circle{2}}
\put(0,0){\circle{2}}
\put(4,4){$U$}
\put(98,0){$Y$}
\put(98,10){$X$}
\multiput(-3,-3)(1,0){16}{.}
\multiput(-3,12)(1,0){16}{.}
\multiput(-3,-3)(0,1){16}{.}
\multiput(12,-3)(0,1){16}{.}
\multiput(17,-3)(1,0){80}{.}
\multiput(17,2)(1,0){80}{.}
\multiput(17,7)(1,0){80}{.}
\multiput(17,12)(1,0){80}{.}
\multiput(17,-3)(0,1){6}{.}
\multiput(17,7)(0,1){6}{.}
\multiput(96,-3)(0,1){6}{.}
\multiput(96,7)(0,1){6}{.}
\end{picture}
\end{center}
\caption{\label{f5}The complement of $\Delta(G)$ and a resolving set}
\end{figure}
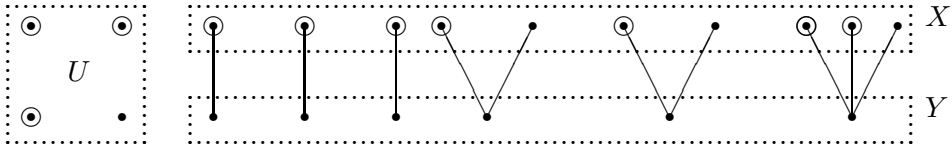

We let the stars be $S_1,\ldots,S_r$, where $r=n_3$; let their central vertices
be $y_1,\ldots,y_r$, and the neighbourhoods of these vertices be $N_1,\ldots
N_r$. (These sets form a partition of $X$.) Note that some stars may consist
of a single edge, in which case which vertex is in $X$ and which in $Y$ are
not determined. Let $s$ be the number of trees with more than two vertices.

\begin{theorem}\label{3.4} The metric dimension of a Lewis graph $\Gamma$ with the parameters just defined is $n_1+n_2-s-1$ if $n_1>0$, or $n_2-s$ if $n_1=0.$ 
\end{theorem}
\begin{proof}
Let $W$ be a resolving set.

We begin with an observation. Two vertices $v$ and $w$ in $\Gamma$ are
\emph{twins} if they have the same neighbours (possibly except for one another).(We have more to say about the twin relation in the next section.)
If $v$ and $w$ are twins, the transposition interchanging them and fixing all
other vertices is an automorphism of $\Gamma$. So there cannot be a resolving
set containing neither of the two vertices. Similarly, if a set of vertices
are mutual twins, then a resolving set must contain all or all but one of them.
Note that the relation of being twins is the same in a graph and its
complement, so we can work in the simpler graph $\bar\Gamma$.

Now the vertices in $U$ are pairwise twins; so $W$ contains at least $n_1-1$ of
them (if $n_1>0$).

Similarly, suppose that $N_i>1$ for some $i$. Then the vertices in $N_i$ are
pairwise twins, and so $W$ must contain at least $|N_i|-1$ of them.

Finally, if $y\in Y$ has a single neighbour $x$ in $X$ in $\bar\Gamma$, then
$x$ and $y$ are twins, so at least one of them lies in $W$; we may assume 
without loss that $x\in W$.

So $W$ contains at least $n_1-1$ vertices of $U$ (if $n_1>0$) and at least
$n_2-s$ vertices in $X$ (we can leave out at most one vertex from each of the
$s$ sets $N_i$ with $|N_i|>1$). This gives the number in the Theorem as a lower
bound.

However, it is straightforward to check that, if $W_0$ consists just of these
vertices, then $W_0$ is a resolving set: the vertex in $U\setminus W$ is
at distance~$1$ from all vertices in $W_0$; a vertex in $N_i$ with $i>1$ is
at distance one from all except $y_i$, from which it has distance $2$; and
if $N_i=\{x_j\}$, then $x_j$ has distance one from all except $y_j$, from
which it has distance~$2$. So $W_0$ is a resolving set, and by the first
part of the proof it is of minimal size. The theorem is proved. \qed
\end{proof}

\begin{example}
Consider the graph given in Figure \ref{f2}. It is a character degree graph $\Delta(G)$ for some solvable group $G$ of Fitting height 2 (see \cite[p: 503]{BL}). In this graph, $X=\{v_3,v_4,v_5\}$, $Y=\{v_1,v_2 \}$, and $U=\{v_6\}$, and hence $n_1=1$, $n_2=3$, and $n_3=2$. Note that $W=\{v_4, v_5, v_6\}$ is a minimum resolving set of $\Delta(G)$ and so the metric dimension is $n_1+n_2-1=1+3-1=3$.
\end{example}

\begin{figure}[ht]
\centering
\begin{tikzpicture} [scale=1.5,colorstyle/.style={circle, draw=black!100,fill=black!100, thick, inner sep=0pt, minimum size=2mm}]          
\node at (0,1)[colorstyle,label=above:$v_1$]{};
\node at (2,0)[colorstyle,label=right:$v_2$]{};
\node at (1,-1)[colorstyle,label=right:$v_3$]{};
\node at (-1,-1)[colorstyle,label=left:$v_4$]{};
\node at (-2,0)[colorstyle,label=left:$v_5$]{};
\node at (0,0)[colorstyle,label=below:$v_6$]{};
\draw[thick] (0,1) --(2,0)--(1,-1)--(-1,-1)--(-2,0)--(0,0)--(0,1)--(-2,0);
\draw[thick] (-1,-1) --(2,0);
\draw[thick] (1,-1) --(-2,0);
\draw[thick] (-1,-1) --(0,0);
\draw[thick] (1,-1) --(0,0);
\draw[thick] (0,0) --(2,0); 
\end{tikzpicture} 
\caption{Character degree graph of a solvable group with six vertices and Fitting height 2}
\label{f2}
\end{figure}
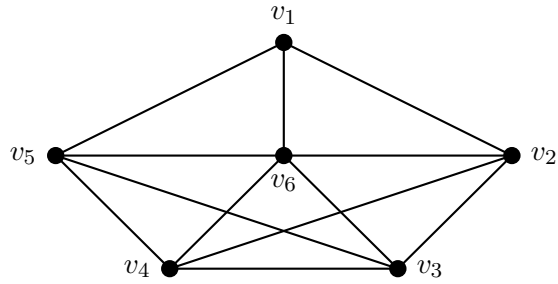

\section{Base size and adjacency dimension}\label{SecBase}
\label{s:further}

Two further invariants related to metric dimension have been studied; we look
briefly at these for character degree graphs, referring to \cite{BC} for a
summary of their properties for general graphs.

\begin{definition}
Let $\Gamma$ be a graph. The \textbf{base size} $b(\Gamma)$ of $\Gamma$ is the
smallest size of a set of vertices such that only the identity automorphism
fixes these vertices point wise. (Such a set of vertices is called a 
\textbf{base} for $\Gamma$.)
\end{definition}

Base size is an important invariant in computational group theory, introduced
by Sims~\cite{Sims} in 1971. It was introduced into graph theory several times,
by Erwin and Harary~\cite{EH} in 2006 by the name \textbf{fixing number}, by
Boutin~\cite{Boutin} in the same year who called it \textbf{determining number},
and by Fijav\v{z} and Mohar~\cite{FM} in 2004 under the name \textbf{rigidity
index}. It was used by Babai in his elementary bound for the orders of
primitive permutation groups~\cite{Babai2}.

\begin{definition}
The \textbf{adjacency dimension} $\adim(\Gamma)$ of $\Gamma$ is the metric
dimension of the metric space on the vertex set of $\Gamma$ in which the
distance of two distinct vertices $v$ and $w$ is~$1$ if $v$ and $w$ are 
adjacent, $2$ otherwise. A resolving set for this metric is an
\textbf{adjacency resolving set}.
\end{definition}

Though this parameter seems more natural than
metric dimension, it was only introduced in 2012, by Jannesari and
Omoomi~\cite{JO}, though adjacency resolving sets had been used earlier by
Babai~\cite{Babai1} under the name \textbf{distinguishing sets}. A survey of
this graph parameter is given by Bermudo~\textit{et al.}~\cite{BRRS}.

Thus, $S$ is an adjacency resolving set if the sets $S\cap N(v)$ are all
distinct for $v\notin S$. In particular, it is a resolving set if $\Gamma$ has
diameter~$2$.

Now it is clear that an adjacency resolving set is a resolving set, and a
resolving set is a base. So we have
\[b(\Gamma)\le\dim(\Gamma)\le\adim(\Gamma).\]

To analyse further, we look more closely at twin vertices in a graph. These
were mentioned in the preceding section; now we treat them more formally.

\begin{definition}
Two vertices $v$ and $w$ are \textbf{twins} if they have the same neighbours,
except possibly for one another. If twins $v$ and $w$ are joined, so that they
have the same closed neighbourhood, we call them \textbf{closed twins};
otherwise, they have the same open neighbourhood, and we call them
\textbf{open twins}. Note that, if vertices are open twins in $\Gamma$,
then they are closed twins in the complement of $\Gamma$, and \emph{vice versa}.
\end{definition}

\begin{proposition}
Being equal or twins is an equivalence relation on the vertex set of a graph,
and all pairs within an equivalence class are the same kind of twins (closed
or open).
\end{proposition}

\begin{proof}
We have to show that it is not possible for $u$ and $v$ to be open twins and
$v$ and $w$ closed twins. If this occurs, then $v$ is joined to $w$ but not $u$.
Now $u$ and $v$ are twins so $u$ is not joined to $w$; but $v$ and $w$ are
twins, so $w$ is joined to $u$, a contradiction.

Clearly being equal or open twins is an equivalence relation, and being equal
or open twins is an equivalence relation; and equivalence classes of size
greater than $1$ are disjoint. \qed
\end{proof}

\begin{theorem}
Let $\Gamma$ be a connected graph on $n$ vertices with the property that every
vertex in $\Gamma$ has a twin. Then $b(\Gamma)=\dim(\Gamma)=\adim(\Gamma)=n-r$,
where $r$ is the number of equivalence classes under the twin relation.
\end{theorem}

\begin{proof}
If two vertices $u$ and $v$ are twins, then they can be interchanged by an
automorphism fixing all other vertices. So a base must contain one of them.
It follows that a base must contain all but at most one vertex in each twin
equivalence class, so that $b(\Gamma)\ge n-r$.

Now let $S$ be a set consisting of all but one vertex from each twin
equivalence class. We show that $S$ is an adjacency resolving set. Suppose
that $u$ and $v$ are two vertices not in $S$ which have the same neighbours
in $S$. If $u$ is joined to $w$, then it is joined to every vertex in the twin
class of $w$. Since $S$ meets every twin class, we see that $u$ and $v$ have
the same neighbours, and so are twins. But this contradicts the assumption 
that $S$ contains all but one vertex from every twin class. So our claim is
proved, and we conclude that $\adim(S)\le n-r$.

Hence
\[n-r\le b(\Gamma)\le\dim(\Gamma)\le\adim(\Gamma)=n-r,\]
so we have equality throughout. \qed
\end{proof}

As an application, we prove a stronger version of Theorem~\ref{3.1}.

\begin{proposition}
If the character degree graph $\Delta(G)$ of a solvable group of even order
$n>4$ is regular, then its base size, dimension and adjacency dimension are
both equal to $n/2$.	
\end{proposition}

For its complement is a $1$-factor, so every vertex has a twin, and there are
$n/2$ twin classes.

The graphs shown in Figure \ref{f3} has base size~$3$ (the 
automorphism group is $S_3\times S_2$, and $\{v_1,v_2,v_5\}$ is a base),
equal to its metric dimension. The graphs in Figures~\ref{f2} and \ref{f4}
have automorphism group of order $2$, and so base size~$1$, smaller than
their metric dimensions (which are $3$ in both cases). 

Examples in~\cite{BC} show that, for arbitrary graphs, the gap between base
size and metric dimension can be arbitrarily large.

\begin{question}
Let $G$ be a finite solvable group and $\Delta(G)$ its character degree graph.
Is $\dim(\Delta(G))$ bounded above by a function of $b(\Delta(G))$?
\end{question}

We will prove this in the case of groups of Fitting height~$2$.

\begin{theorem}
Let $G$ be a finite solvable group of Fitting height~$2$, and $\Delta(G)$ its
character degree graph. Then $b(\Delta(G)=\dim(\Delta(G))=\adim(\Delta(G))$.
\end{theorem}

\begin{proof}
We computed the dimension in Theorem~\ref{3.4}. Since the diameter is $2$
by the theorem of Lewis, this is equal to the adjacency dimension. Also our
remark about twins in the proof shows that a base must contain at least one
vertex from each equivalence class of twins; so the set constructed in that
theorem is also a base. \qed
\end{proof}


\begin{thebibliography}{99}

\bibitem{Babai1}
L. Babai,
\textit{On the complexity of canonical labeling of strongly regular graphs},
SIAM J. Comput. \textbf{9} (1980), 212--216.

\bibitem{Babai2}
L. Babai,
\textit{On the order of uniprimitive permutation groups},
Ann. Math. \textbf{113} (1981), 553--568.

\bibitem{BC}
Robert F. Bailey and Peter J. Cameron,
\textit{Base size, metric dimension, and other invariants of groups and graphs},
Bull. London Math. Soc. \textbf{43} (2011), 209--242

\bibitem{BRRS}
Sergio Bermudo, Jos\'e M. Rodr\'{\i}guez, Juan A. Rodr\'{\i}guez-Vel\'azquez
and Jos\'e M. Sigarreta, 
\textit{The adjacency dimension of graphs},
Ars Mathematica Contemporanea \textbf{22} (2022), \#P3.02.

\bibitem{BissLaub} M. W. Bissler and J. Laubacher,  \textit{Classifying families of character degree graphs of solvable groups}, Int. J. Group Theory  \textbf{8} (2019), no. 4, 37--46.


\bibitem{BL} M. W. Bissler, J. Laubacher, and M. L. Lewis, \textit{Classifying character degree graphs with six vertices}, Beitr. Algebra Geom. \textbf{60} (2019), 499--511.

\bibitem{BM} J. A. Bondy and U. S. R. Murty, \textit{Graph Theory}, Grad. Texts in Math., vol. 244, Springer, New York, 2008.

\bibitem{Boutin}
D. L. Boutin, 
\textit{Identifying graph automorphisms using determining sets},
Electron. J. Combin. 13 (2006), \#R78.

\bibitem{CE} G. Chartrand, L. Eroh, M. A. Johnson, and O. R. Oellermann, \textit{Resolvability in graphs and the metric dimension of a graph}, Discrete Appl. Math. \textbf{105} (2000), 99--113.

\bibitem{LJ} S. DeGroot, J. Laubacher, and M. Medwid, \textit{On prime character degree graphs occurring within a family of graphs (ii)}, Comm. Algebra \textbf{50} (2022), no. 8, 3307--3319.

\bibitem{EI} M. Ebrahimi, A. Iranmanesh, and M. A. Hosseinzadeh, \textit{Hamiltonian character graphs}, J. Algebra  \textbf{428} (2015), no. 6, 54--66.

\bibitem{EH}
D. Erwin and F. Harary, 
\textit{Destroying automorphisms by ﬁxing nodes},
Discrete Math. \textbf{306} (2006), 3244--3252.

\bibitem{FM}
G. Fijav\v{z} and B. Mohar,
\textit{Rigidity and separation indices of Paley graphs},
Discrete Math. \textbf{289} (2004), 157--161.

\bibitem{HHH} R. Hafezieh, M. A. Hosseinzadeh, S. Hossein-Zadeh, and A. Iranmanesh,  \textit{On cut vertices and eigenvalues of character graphs of solvable groups}, Discrete Appl. Math. \textbf{303} (2021), 86--93.


\bibitem{HALL} M. Hall Jr.,  \textit{The Theory of Groups}, Macmillan, New York, 1959.

\bibitem{Har} F. Harary and R. A. Melter, \textit{On the metric dimension of a graph}, Ars Combin. \textbf{2} (1976), 191--195.

\bibitem{JO}
Mohsen Jannesari and Behnaz Omoomi,
\textit{The metric dimension of the lexicographic product of graphs},
Discrete Math. \textbf{312} (2012), 3349--3356.

\bibitem{LM}  J. Laubacher and M. Medwid, \textit{On prime character degree graphs occurring within a family of graphs}, Comm. Algebra \textbf{49} (2021), no. 4, 1534--1547.


\bibitem{Lewis2}  M. L. Lewis, \textit{Solvable groups with character degree graphs having 5 vertices and diameter 3}, Comm. Algebra   \textbf{30} (2002), 5485--5503.


\bibitem{Lewis3} M. L. Lewis, \textit{A solvable group whose character degree graph has diameter 3}, Proc. Amer. Math. Soc. \textbf{130} (2002), no. 3, 625--630.

\bibitem{Lewis4}M. L. Lewis, \textit{Character degree graphs of solvable groups of fitting height 2}, Canad. Math. Bull. \textbf{49} (2006), no. 1, 127--133.


\bibitem{Lewis5} M. L. Lewis, \textit{An overview of graphs associated with character degrees and conjugacy class sizes in finite groups}, Rocky Mountain J. Math. \textbf{38} (2008), no. 1, 175--211.




\bibitem{Lewis6} M. L. Lewis and Q. Meng, \textit{Solvable groups whose prime divisor character degree graphs are 1-connected}, Monatsh. Math.  \textbf{190} (2019), 541--548.

\bibitem{Manz1}  O. Manz, \textit{Degree problems II: $\pi$-separable character degrees}, Comm. Algebra \textbf{13} (1985), 2421--2431.


\bibitem{Manz2} O. Manz, R. Staszewski, and W. Willems, \textit{On the number of components of a graph related to character degrees}, Proc. Amer. Math. Soc. \textbf{103} (1988), no. 1, 3137.

\bibitem{MZ}  C. P. Morresi Zuccari, \textit{Regular character degree graphs}, J. Algebra \textbf{411} (2014), 215--224.


\bibitem {PSR}  S. Pirzada and R. Raja, \textit{On the metric dimension of a zero-divisor graph}, Comm. Algebra \textbf{45} (2016), no. 4, 1399--1408.


\bibitem{Sass} C. B. Sass, \textit{Character degree graphs of solvable groups with diameter three}, J. Group Theory \textbf{19} (2016), no. 6, 1097--1127.

\bibitem{Sims}
C. C. Sims,
\textit{Determining the conjugacy classes of a permutation group},
Computers in algebra and number theory (eds G. Birkhoﬀ and M. Hall, Jr.),
American Mathematical Society, Providence, RI, 1971, 191--195.

\bibitem{SST}  G. Sivanesan, C. Selvaraj, and T. Tamizh Chelvam, \textit{Eulerian character degree graphs of solvable groups}, AKCE Int. J. Graphs Combin. \textbf{21} (2024), no. 2, 161--166.

\bibitem{Manz3} Th. R. Wolf, O. Manz, and W. Willems, \textit{The diameter of the character degree graph}, J. Reine Angew. Math. \textbf{402} (1989), 181198.

\end{thebibliography}
\end{document}